\documentclass[11pt,letterpaper]{article}%
\usepackage{amsfonts}
\usepackage{amsmath,amssymb,amsthm}%
\usepackage{amsmath}%
\usepackage{amssymb}%
\usepackage{graphicx}
\newtheorem{theorem}{Theorem}

\newtheorem{lemma}[theorem]{Lemma}

\begin{document}

\title{A regularization algorithm for matrices of bilinear and sesquilinear forms\footnotetext{This is the authors' version of a work that was published in Linear Algebra Appl. 412 (2006) 380--395.}}

\author{Roger A. Horn\\Department of Mathematics, University of Utah\\Salt Lake City, Utah 84103, rhorn@math.utah.edu
\and Vladimir V. Sergeichuk\thanks{The research was started while this author was
visiting the University of Utah supported by NSF grant DMS-0070503.}\\Institute of Mathematics, Tereshchenkivska 3\\Kiev, Ukraine, sergeich@imath.kiev.ua}
\date{}
\maketitle

\begin{abstract}
Over a field or skew field $\mathbb{F}$ with an involution $a\mapsto
\widetilde{a}$ (possibly the identity involution), each singular square matrix
$A$ is *congruent to a direct sum
\[
S^{\ast}AS=B\oplus J_{n_{1}}\oplus\dots\oplus J_{n_{p}},\qquad1\leq n_{1}%
\leq\cdots\leq n_{p},
\]
in which $S$ is nonsingular and $S^{\ast}=\widetilde{S}^{T}$; $B$ is
nonsingular and is determined by $A$ up to *congruence; and the $n_{i}%
$-by-$n_{i}$ singular Jordan blocks $J_{n_{i}}$ and their multiplicities are
uniquely determined by $A$. We give a regularization algorithm that needs only
elementary row operations to construct such a decomposition. If $\mathbb{F}%
=\mathbb{C}$ (respectively, $\mathbb{F}=\mathbb{R}$), we exhibit a
regularization algorithm that uses only unitary (respectively, real
orthogonal) transformations and a reduced form that can be achieved via a
unitary *congruence or congruence (respectively, a real orthogonal
congruence). The selfadjoint matrix pencil $A+\lambda A^{\ast}$ is decomposed
by our regularization algorithm into the direct sum
\[
S^{\ast}(A+\lambda A^{\ast})S=(B+\lambda B^{\ast})\oplus(J_{n_{1}}+\lambda
J_{n_{1}}^{\ast})\oplus\dots\oplus(J_{n_{p}}+\lambda J_{n_{p}}^{\ast})
\]
with selfdajoint summands.



\end{abstract}

\section{Introduction\label{Section Introduction}}

\label{s1}

All of the matrices that we consider are over a field or skew field
$\mathbb{F}$ with an involution $a\mapsto\widetilde{a}$, that is, a bijection
on $\mathbb{F}$ such that
\[
\widetilde{a+b}=\widetilde{a}+\widetilde{b},\qquad\widetilde{ab}=\widetilde
{b}\widetilde{a},\qquad\widetilde{\widetilde{a}}=a.
\]
If $\mathbb{F}$ is a field, the identity mapping $a\mapsto a$ on $\mathbb{F}$
is always an involution; over the complex field, complex conjugation
$a\mapsto\bar{a}$ is an involution. We refer to $\widetilde{a}$ as the
\emph{conjugate} of $a$.

The entry-wise conjugate of the transpose of a matrix $A=[a_{ij}]$ is denoted
by
\[
A^{\ast}=\widetilde{A}^{T}=[\widetilde{a}_{ji}]\text{.}%
\]
If there is a square nonsingular matrix $S$ such that $S^{\ast}AS=B$, then $A$
and $B$ are said to be *\emph{\negthinspace congruent}; if the involution on
$\mathbb{F}$ is the identity, i.e., $S^{\ast}=S^{T}$ and $S^{\ast}%
AS=S^{T}AS=B$, we say that $A$ and $B$ are \emph{congruent}. Congruence of
matrices (sometimes called \emph{T}-\emph{\negthinspace congruence}) is
therefore a special type of *congruence in which the involution is the
identity. Over the complex field with complex conjugation as the involution,
*congruence is sometimes called \emph{conjunctivity}. If $A$ is nonsingular,
we write $A^{-\ast}=(A^{\ast})^{-1}$.

Let
\[
J_{n}=%
\begin{bmatrix}
0 & 1 &  & 0\\
& 0 & \ddots & \\
&  & \ddots & 1\\
0 &  &  & 0
\end{bmatrix}
\]
denote the $n$-by-$n$ singular Jordan block.

For any $m$-by-$n$ matrix $A$ (that is, $A\in\mathbb{F}^{m\times n}$) we write
$N(A):=\{x\in\mathbb{F}^{n}:Ax=0\}$ (the \emph{null space} of $A$) and denote
its dimension by $\dim N(A)=\operatorname{nullity}A$. If $A$ is square, we
let
\[
A^{[k]}:=A\oplus\cdots\oplus A\text{\ (}k\text{ times).}%
\]

In Section \ref{Section Algorithm} we describe a constructive
\emph{regularization algorithm} that determines a \emph{regularizing
decomposition}
\begin{equation}
B\oplus J_{n_{1}}\oplus\cdots\oplus J_{n_{p}}\text{,}\quad\text{ }B\text{
nonsingular and }1\leq n_{1}\leq\cdots\leq n_{p}\label{c2}%
\end{equation}
to which a given square singular matrix $A$ is *congruent. The *congruence
class of $B$ (the \emph{regular part} of $A$ under *congruence) and the sizes
and multiplicities of the direct summands $J_{n_{1}},\dots,J_{n_{p}}$ (the
\emph{singular part} of $A$ under *congruence) are all uniquely determined by
the *congruence class of $A$. If $\mathbb{F}=\mathbb{C}$ (respectively,
$\mathbb{F}=\mathbb{R}$), the regularizing decomposition (\ref{c2}) can be
determined using only unitary (respectively, real orthogonal) transformations.
Our proof of the existence and uniqueness of the regularizing decomposition
(\ref{c2}) uses two geometric *congruence invariants that we discuss in
Section \ref{Section Invariants}: $\dim N(A)$ and $\dim(N(A^{\ast})\cap N(A))$.

In Section \ref{Section Decomposition} we exhibit a canonical \emph{sparse
form} that is *congruent to $A$ and determines the sizes and multiplicities of
the nilpotent direct summands in the regularizing decomposition (\ref{c2}).
The essential parameters of the sparse form are identical to those produced by
our regularization algorithm, which verifies the validity of the algorithm.
When $\mathbb{F}=\mathbb{C}$ or $\mathbb{R}$, we describe a reduced form
related to the canonical sparse form that can be achieved using only unitary
*congruences or $T$-congruences.

The regularization algorithm reduces the problem of determining a *congruence
canonical form to the nonsingular case. A complete set of *congruence
canonical forms (up to classification of Hermitian forms) when $\mathbb{F}$ is
a field with characteristic not equal to two is given in \cite[Theorem
3]{ser1}; see also \cite[Theorem 2]{hor-ser}. A nonalgorithmic reduction to
the nonsingular case was given by Gabriel for bilinear forms \cite{gab}; his
method was extended in \cite{Riehm} to sesquilinear forms, and in \cite{ser1}
to systems of sesquilinear forms and linear mappings. The form of the
regularizing decomposition (\ref{c2}) is implicit in the statement of
Proposition 3.1 in \cite{DSZ} when $\mathbb{F}$ is a field and the involution
is the identity; the construction employed in its proof does not suggest a
simple algorithm for identifying the parameters in (\ref{c2}).

If $A,B\in\mathbb{F}^{m\times n}$, then the polynomial matrix $A+\lambda B$ is
called a \emph{matrix pencil}. Two matrix pencils $A+\lambda B$ and
$A^{\prime}+\lambda B^{\prime}$ are said to be \emph{strictly equivalent} if
there exist nonsingular matrices $S$ and $R$ such that $S(A+\lambda
B)R=A^{\prime}+\lambda B^{\prime}$. Van Dooren \cite{doo} has given an
algorithm that uses only unitary transformations and for each complex matrix
pencil $A+\lambda B$ produces a strictly equivalent pencil
\begin{equation}
(C+\lambda D)\oplus(M_{1}+\lambda N_{1})\oplus\cdots\oplus(M_{l}+\lambda
N_{l})\label{eq2}%
\end{equation}
in which $C$ and $D$ are nonsingular constituents of the \emph{regular part}
$C+\lambda D$ of the Kronecker canonical form of $A+\lambda B$; each
$M_{i}+\lambda N_{i}$ is a singular direct summand of that canonical form
\cite[Section XII, Theorem 5]{gan}. Each $M_{i}+\lambda N_{i}$ has the form
\[
I_{n}+\lambda J_{n},\quad J_{n}+\lambda I_{n},\quad F_{n}+\lambda G_{n}%
,\quad\text{or}\quad G_{n}^{T}+\lambda F_{n}^{T}%
\]
for some $n$, in which
\[
F_{n}=%
\begin{bmatrix}
1 & 0 &  & 0\\
& \ddots & \ddots & \\
0 &  & 1 & 0
\end{bmatrix}
\quad\text{and}\quad G_{n}=%
\begin{bmatrix}
0 & 1 &  & 0\\
& \ddots & \ddots & \\
0 &  & 0 & 1
\end{bmatrix}
\quad\text{are $(n-1)$-by-$n$}.
\]
The direct sum \eqref{eq2} is a \emph{regularizing decomposition} of
$A+\lambda B$; $C+\lambda D$ is the \emph{regular part} of $A+\lambda B$. Van
Dooren's algorithm was extended to cycles of linear mappings with arbitrary
orientation of arrows in \cite{vvs2004}.

If Van Dooren's algorithm is used to construct a regularizing decomposition of
a *selfadjoint matrix pencil $A+\lambda A^{\ast}$, the regular part produced
need not be *selfadjoint. However, the regularizing decomposition of
$A+\lambda A^{\ast}$ that we describe in Section \ref{Section Pencil} always
produces a *selfadjoint regular part.

For any nonnegative integers $m$ and $n$, we denote the $m$-by-$n$ zero matrix
by $0_{mn}$, or by $0_{m}$ if $m=n$. The $n$-by-zero matrix $0_{n0}$ is
understood to represent the linear mapping $0\rightarrow{\mathbb{F}}^{n}$; the
zero-by-$n$ matrix $0_{0n}$ represents the linear mapping ${\mathbb{F}}%
^{n}\rightarrow0$; the zero-by-zero matrix $0_{0}$ represents the linear
mapping $0\rightarrow0$. For every $p\times q$ matrix $M_{pq}$ we have
\[
M_{pq}\oplus0_{m0}=%
\begin{bmatrix}
M_{pq} & 0\\
0 & 0_{m0}%
\end{bmatrix}
=%
\begin{bmatrix}
M_{pq} & 0_{p0}\\
0_{mq} & 0_{m0}%
\end{bmatrix}
=%
\begin{bmatrix}
M_{pq}\\
0_{mq}%
\end{bmatrix}
\]
and
\[
M_{pq}\oplus0_{0n}=%
\begin{bmatrix}
M_{pq} & 0\\
0 & 0_{0n}%
\end{bmatrix}
=%
\begin{bmatrix}
M_{pq} & 0_{pn}\\
0_{0q} & 0_{0n}%
\end{bmatrix}
=%
\begin{bmatrix}
M_{pq} & 0_{pn}%
\end{bmatrix}
.
\]
In particular,
\[
0_{p0}\oplus0_{0q}=0_{pq}%
\]
and $J_{n}^{[0]}=0_{0}$. Consistent with the definition of singularity, our
convention is that \emph{a zero-by-zero matrix is nonsingular}.

\section{The regularization algorithm\label{Section Algorithm}}

The first stage in our \emph{regularization algorithm} for a singular square
matrix $A$ is to reduce it by *congruence transformations in two steps that
construct a smaller matrix $A_{(1)}$ and integers $m_{1}$ and $m_{2}$ as follows:

\begin{description}
\item[Step 1] Choose a nonsingular $S$ such that the top rows of $SA$ are
linearly independent and the bottom $m_{1}$ rows are zero, then form
$(SA)S^{\ast}$ and partition it so that the upper left block is square:
\begin{align}
A\, &  \longmapsto\,SA=%
\begin{bmatrix}
A^{\prime}\\
0
\end{bmatrix}
\quad%
\begin{matrix}
\text{($S$ is nonsingular and the rows of}\\
\text{$A^{\prime}$ are linearly independent)}%
\end{matrix}
\nonumber\\
&  \longmapsto SAS^{\ast}=\,%
\begin{bmatrix}
A^{\prime}S^{\ast}\\
0
\end{bmatrix}
=\left[
\begin{array}
[c]{l|l}%
M & N\\\hline
0 & \;0_{m_{1}}%
\end{array}
\right]  \;%
\begin{array}
[c]{c}%
\text{($S$ is the same and}\\
\text{$M$ is square)}%
\end{array}
\text{ }\label{Step1}%
\end{align}
The integer $m_{1}$ is the nullity of $A$.

\item[Step 2] Choose a nonsingular $R$ such that the top rows of $RN$ are zero
and the bottom $m_{2}$ rows are linearly independent:
\begin{equation}
RN=\left[
\begin{array}
[c]{c}%
0\\
E
\end{array}
\right]  \quad%
\begin{array}
[c]{c}%
\text{(}R\text{ is nonsingular and the rows of}\\
\text{ }E\text{ are linearly independent)}%
\end{array}
\text{ }\label{Step2a}%
\end{equation}
The integer $m_{2}$ is the rank of $N$. Now perform a *congruence of $S^{\ast
}AS$ with $R\oplus I$:
\begin{align}
\left[
\begin{array}
[c]{c|c}%
M & N\\\hline
0 & 0
\end{array}
\right]   & \longmapsto\,(R\oplus I)\left[
\begin{array}
[c]{c|c}%
M & N\\\hline
0 & 0
\end{array}
\right]  (R\oplus I)^{\ast}\label{Step2}\\
& =\left[
\begin{array}
[c]{c|c}%
RMR^{\ast} & RN\\\hline
0 & 0
\end{array}
\right]  =\left[
\begin{array}
[c]{cc|c}%
A_{(1)} & B & 0\\
C & D & E\\\hline
\multicolumn{2}{c|}{0} & 0_{m_{1}}%
\end{array}
\right]  \hspace{-0.13in}%
\begin{array}
[c]{l}%
\\
\}m_{2}\\
\}m_{1}%
\end{array}
\label{Step2b}%
\end{align}
The block $RMR^{\ast}$ has been partitioned so that $D$ is $m_{2}$-by-$m_{2}
$. The size of the square matrix $A_{(1)}$ is strictly less than that of $A$.
\end{description}

If $A_{(1)}$ is nonsingular, the algorithm terminates. If $A_{(1)}$ is
singular, the second stage of the regularization algorithm is to perform the
two *congruences (\ref{Step1}) and (\ref{Step2}) on it and obtain integers
$m_{3}$ (the nullity of $A_{(1)}$) and $m_{4}$, and a square matrix $A_{(2)}$
whose size is strictly less than that of $A_{(1)}$.

The regularization algorithm proceeds from stage $k$ to stage $k+1$ by
performing the two *congruences (\ref{Step1}) and (\ref{Step2}) on the
singular square matrix $A_{(k-1)}$ to obtain $m_{2k-1}$, $m_{2k}$, and
$A_{(k)}$.\ When the algorithm terminates at stage $\tau$ with a square matrix
$A_{(\tau)}$ that is nonsingular, we have in hand a non-increasing sequence of
integers $m_{1}\geq m_{2}\geq\cdots\geq m_{2\tau-1}\geq m_{2\tau}\geq0$ and a
nonsingular matrix $A_{(\tau)}$. Our main result is that these data determine
the singular part of $A$ under *congruence as well as the *congruence class of
the regular part according to the following rule:

\begin{theorem}
\label{theorem regularization algorithm} Let $A$ be a given square singular
matrix over $\mathbb{F}$ and apply the regularization algorithm to it. Then
$A$ is *congruent to $A_{(\tau)}\oplus M$, in which $A_{(\tau)}$ is
nonsingular and%
\begin{equation}
M=J_{1}^{[m_{1}-m_{2}]}\oplus J_{2}^{[m_{2}-m_{3}]}\oplus J_{3}^{[m_{3}%
-m_{4}]}\oplus\cdots\oplus J_{2\tau-1}^{[m_{2\tau-1}-m_{2\tau}]}\oplus
J_{2\tau}^{[m_{2\tau}]}\text{.}\label{RegDecomp1}%
\end{equation}
The integers $m_{1}\geq m_{2}\geq\cdots\geq m_{2\tau-1}\geq m_{2\tau}\geq0$,
as well as the *congruence class of $A_{(\tau)}$, are uniquely determined by
the *congruence class of $A$.
\end{theorem}

In the next section we offer a geometric interpretation for the integers
$m_{i}$ in (\ref{RegDecomp1}) and explain why they and the *congruence class
of each of the square matrices $A_{(k)}$ produced by the regularization
algorithm are *congruence invariants of $A$. Implicit in the regularization
algorithm are certain reductions of $A$ by *congruences that we refine in
order to explain why the regularizing decomposition in (\ref{RegDecomp1}) is valid.

The nonsingular matrices $S$ and $R$ in the two *congruence steps of the
regularization algorithm can always be constructed with elementary row
operations. For the complex (respectively, real) field, it can be useful for
numerical implementation to know that $S$ and $R$ may be chosen to be unitary
(respectively, real orthogonal).

\begin{theorem}
\label{Theorem Unitary transformations}Let $A$ be a given square singular
complex (respectively, real) matrix. The regularizing decomposition
(\ref{RegDecomp1}) of $A$ can be determined using only unitary (respectively,
real orthogonal) transformations.
\end{theorem}

\begin{proof}
(a) Suppose $\mathbb{F=C}$ with complex conjugation as the involution. Let
$A=U^{\ast}\Sigma Z$ be a singular value decomposition in which $\Sigma
=\Sigma_{1}\oplus0_{m_{1}}$, $\Sigma_{1}$ is positive diagonal, and $U$ and
$Z$ are unitary. The choice $S=U$ achieves the required reduction in Step 1.
In Step 2, let $N=\hat{V}^{\ast}\hat{\Sigma}W$ be a singular value
decomposition in which $\hat{V}$ and $W$ are unitary, $\hat{\Sigma}%
=\hat{\Sigma}_{1}\oplus0$, and $\hat{\Sigma}_{1}$ is positive diagonal and
$m_{2}$-by-$m_{2}$. Let%
\[
P=\left[
\begin{array}
[c]{ccc}
&  & 1\\
&
{\text{\raisebox{-2.2pt}{$\cdot\,$} \raisebox{1.7pt}{$\cdot$}\raisebox
{5.6pt}{$\,\cdot$}}}%
& \\
1 &  &
\end{array}
\right]
\]
be the reversal matrix whose size is the same as that of $\hat{V}$. Then
$N=(P\hat{V})^{\ast}(P\hat{\Sigma})W$, $(P\hat{\Sigma})W$ has the block form
(\ref{Step2a}), and $V:=P\hat{V}$ is unitary, so we may take $R=V$ in Step 2.
Thus, $A$ is unitarily *congruent (unitarily similar) to a block matrix of the
form (\ref{Step2b}) in which $D$ is square and each of $E$ and $[A_{(1)}\;B]$
has linearly independent rows.

(b) Suppose $\mathbb{F=C}$ with the identity involution. In Step 1, choose
$S=\bar{U}$ from (a). In Step 2, choose $R=\bar{V}$ from (a). Thus, $A$ is
unitarily $T$-congruent to a block matrix of the form (\ref{Step2b}) in which
$D$ is square and each of $E$ and $[A_{(1)}\;B]$ has linearly independent rows.

(c) Suppose $\mathbb{F=R}$ with the identity involution. Proceed as in (a),
choosing $U$, $\hat{V}$, and $W$ to be real orthogonal in the two singular
value decompositions. Thus, $A$ is real orthogonally congruent to a block
matrix of the form (\ref{Step2b}) in which $D$ is square and each of $E$ and
$[A_{(1)}\;B]$ has linearly independent rows.
\end{proof}

\medskip The regularizing algorithm tells how to construct a sequence of pairs
of transformations of the square matrices $A_{(k)}$ that are sufficient to
determine the regularizing decomposition of $A$. Implicit in these
transformations is a sequence of pairs of *congruences that reduce $A$ in
successive stages. After the first stage, the *congruences reduce $A$ to the
form (\ref{Step2b}). After the second stage, if we were to carry out the
*congruences we would obtain a matrix of the form%
\begin{equation}
\left[
\begin{array}
[c]{ccccc}%
A_{(2)} & \ast & 0 & \ast & 0\\
\ast & \ast & \blacksquare & \ast & 0\\
0 & 0 & 0 & \blacksquare & 0\\
\ast & \ast & \ast & \ast & \blacksquare\\
0 & 0 & 0 & 0 & 0
\end{array}
\right]  \hspace{-0.13in}%
\begin{array}
[c]{l}%
\\
\}m_{4}\\
\}m_{3}\\
\}m_{2}\\
\}m_{1}%
\end{array}
\label{Stage 2}%
\end{equation}
in which the diagonal blocks are square, the * blocks are not necessarily
zero, and each $\blacksquare$ block has linearly independent rows. Theorem
\ref{Theorem Unitary transformations} ensures that if $A$ is complex, then
there are unitary matrices $U$ and $V$ such that each of $U^{\ast}AU$ and
$V^{T}AV$ has the form (\ref{Stage 2}), with possibly different values for the
parameters $m_{i}$. If $A$ is real, there is a real orthogonal $Q$ such that
$Q^{T}AQ$ has the form (\ref{Stage 2}).

\section{*Congruence Invariants and a Reduced Form\label{Section Invariants}}

Throughout this section, $A\in\mathbb{F}^{m\times m}$ and $S$ is a nonsingular
matrix. Of course, $\operatorname{nullity}A=\operatorname{nullity}S^{\ast}AS$,
so nullity is a *congruence invariant. The relationships%
\begin{equation}
N(S^{\ast}AS)=S^{-1}N(A)\quad\text{and\quad}N(S^{\ast}A^{\ast}S)=S^{-1}%
N(A^{\ast})\label{invariant0}%
\end{equation}
between the null spaces of $A$ and $S^{\ast}AS$, and those of $A^{\ast}$ and
$S^{\ast}A^{\ast}S$, imply that%
\begin{equation}
N(S^{\ast}A^{\ast}S)\cap N(S^{\ast}AS)=S^{-1}\left(  N(A^{\ast})\cap
N(A)\right)  \text{.}\label{invariant1}%
\end{equation}
We refer to $\zeta:=$ $\dim N(A^{\ast})\cap N(A)$ as the \emph{*normal
nullity} of $A$. We let $\nu:=\operatorname{nullity}A$, refer to $\kappa
:=\nu-\zeta$ as the \emph{*non-normal nullity} of $A$, and let $\rho
=m-\kappa-\nu$.

It follows from (\ref{invariant0}) and (\ref{invariant1}) that $\nu$, $\zeta$,
$\kappa$, and $\rho$ are *congruence invariants. Because%
\begin{equation}
N(A^{\ast})\cap N(A)=N\left(  \left[
\begin{array}
[c]{c}%
A\\
A^{\ast}%
\end{array}
\right]  \right)  \text{,}\label{invariant2}%
\end{equation}
$\nu$ and $\zeta$ (and hence also $\kappa$ and $\rho$) can be computed using
elementary row operations.

The parameter $m_{1}$ produced by the regularization algorithm is the nullity
of $A$, so it is a *congruence invariant: $m_{1}=\nu$.

The parameter $m_{2}$ produced by the regularization algorithm is the rank of
the block $N$ in (\ref{Step1}). Since $N$ has $m_{1}$ columns and full row
rank, its nullity is $m_{1}-m_{2}$. Suppose $z\in\mathbb{F}^{m_{1}}$ and
$Nz=0$, let $y^{\ast}=[0\;z^{\ast}]$, and let $\mathcal{A}=SAS^{\ast}$ denote
the block matrix in (\ref{Step1}). Then $\mathcal{A}y=0$ and $y^{\ast
}\mathcal{A}=0$ so $\zeta=\dim(N(\mathcal{A}^{\ast})\cap N(\mathcal{A}%
))=\operatorname{nullity}N=m_{1}-m_{2}$ and hence $m_{1}-m_{2}=\zeta$ is the
*normal nullity of $A$. This means that $m_{2}=m_{1}-\zeta=\nu-\zeta=\kappa$
is the *non-normal nullity of $A$, so $m_{2}$ is also a *congruence invariant.

The following lemma ensures that the *congruence class of the square matrix
$A_{\left(  1\right)  }$ in (\ref{Step2b}) is also a *congruence invariant.

\begin{lemma}
\label{lemma reduced form} Suppose that a singular square matrix $A$ is
*\negthinspace congruent to
\[
M=\left[
\begin{array}
[c]{ccc}%
A_{\left(  1\right)  } & B & 0\\
C & D & E\\
0 & 0 & 0_{\nu}%
\end{array}
\right]  \quad\text{and also to }\quad\underline{M}=\left[
\begin{array}
[c]{ccc}%
\underline{A}_{\left(  1\right)  } & \underline{B} & 0\\
\underline{C} & \underline{D} & \underline{E}\\
0 & 0 & 0_{\underline{\nu}}%
\end{array}
\right]  \text{,}%
\]
in which $D$ is $\kappa$-by-$\kappa$, $\underline{D}$ is \underline{$\kappa$%
}-by-\underline{$\kappa$}, and each of $E$, $\underline{E}$, $[A_{\left(
1\right)  }\;B]$, and $[\underline{A}_{\left(  1\right)  }\;\underline{B}]$
has linearly independent rows. Then $\nu=\underline{\nu}$, $\kappa
=\underline{\kappa}$, and $A_{\left(  1\right)  }$ is *congruent to
$\underline{A}_{\left(  1\right)  }$, that is, $\nu$, $\kappa$, $\rho$, and
the *congruence class of the $\rho$-by-$\rho$ matrix $A_{\left(  1\right)  }$
are *congruence invariants of $A$.
\end{lemma}

\begin{proof}
The form of $M$ ensures that $\nu$ is its nullity and that $\kappa$ is its
*non-normal nullity; $\underline{\nu}$ is the nullity of $\underline{M}$ and
$\underline{\kappa}$ is its *non-normal nullity. Since $M$ and $\underline{M}$
are *congruent to $A$ and hence to each other, their nullities and *non-normal
nullities are the same, so $\nu=\underline{\nu}$ and $\kappa=\underline
{\kappa}$.

Let%
\[
\hat{M}=\left[
\begin{array}
[c]{ccc}%
A_{(1)} & B & 0\\
C & D & E
\end{array}
\right]  \text{\ and\ }\underline{\hat{M}}=\left[
\begin{array}
[c]{ccc}%
\underline{A}_{\left(  1\right)  } & \underline{B} & 0\\
\underline{C} & \underline{D} & \underline{E}%
\end{array}
\right]  \text{.}%
\]
If $S=[S_{ij}]_{i,j=1}^{2}$ is nonsingular, $S_{22}$ is $\nu$-by-$\nu$, and
$SMS^{\ast}=$ $\underline{M}$, then
\[
SM=%
\begin{bmatrix}
S_{11} & S_{12}\\
S_{21} & S_{22}%
\end{bmatrix}
\left[
\begin{array}
[c]{c}%
\hat{M}\\
0
\end{array}
\right]  =\left[
\begin{array}
[c]{c}%
\star\\
S_{21}\hat{M}%
\end{array}
\right]  =\left[
\begin{array}
[c]{c}%
\star\\
0
\end{array}
\right]  =\underline{M}S^{-\ast}\text{,}%
\]
so $S_{21}\hat{M}=0$. Full row rank of $\hat{M}$ ensures that $S_{21}=0$ and
hence both $S_{11}$ and $S_{22}$ are nonsingular. If we write $S_{11}%
=[R_{ij}]_{i,j=1}^{2}$, in which $R_{22}$ is $\kappa$-by-$\kappa$, then
equating the $1,2$ blocks of $SMS^{\ast}$ and $\underline{M}$ tells us that
\[
S_{11}\left[
\begin{array}
[c]{c}%
0\\
E
\end{array}
\right]  =\left[
\begin{array}
[c]{c}%
R_{12}E\\
\star
\end{array}
\right]  =\left[
\begin{array}
[c]{c}%
0\\
\star
\end{array}
\right]  =\left[
\begin{array}
[c]{c}%
0\\
\underline{E}%
\end{array}
\right]  \left(  S_{22}\right)  ^{-\ast}\text{,}%
\]
which ensures that $R_{12}=0$, $R_{11}$ and $R_{22}$ are nonsingular, and
$R_{11}A_{\left(  1\right)  }R_{11}^{\ast}=\underline{A}_{\left(  1\right)  }$.
\end{proof}

\medskip Lemma \ref{lemma reduced form}, identification of the parameters in
the first stage of the regularization algorithm as *congruence invariants
($m_{1}=\nu$ and $m_{2}=\kappa$), and an induction argument ensure that at
each stage $k=1,2,...$ of the algorithm the *congruence class of the square
matrix $A_{(k)}$, the integers $m_{2k-1}$ (the nullity of $A_{(k-1)}$) and
$m_{2k}$ (the *non-normal nullity of $A_{(k-1)}$), the number of stages $\tau$
in the algorithm until it terminates, and the *congruence class of the final
nonsingular square matrix $A_{(\tau)}$ are all uniquely determined by the
*congruence class of $A$. All that remains to be shown is that these data
determine the regularizing decomposition of $A$ according to the rule in
Theorem \ref{theorem regularization algorithm}.

The block matrix (\ref{Step2b}) can be reduced to a more sparse form by
*congruence if $m_{2}>0$: the block $E$ may be taken to be $[I_{m_{2}}\;0]$
and the blocks $C$ and $D$ may be taken to be zero. To achieve these
reductions, is is useful to realize that if $A\rightarrow AS$ adds linear
combinations of a set of columns of $A$ with index set $\alpha$ to certain
columns, and if the rows of $A$ with index set $\alpha$ are all zero, then
$S^{\ast}A=A$, so $S^{\ast}AS=AS$.

\begin{lemma}
\label{lemma C=D=0} If a singular square matrix $A$ is *\negthinspace
congruent to a block matrix $\mathcal{A}$ of the form (\ref{Step2b}) in which
$m_{2}>0$, $D$ is $m_{2}$-by-$m_{2}$, and $E$ has linearly independent rows,
then it is *congruent to%
\begin{equation}
\left[
\begin{array}
[c]{ccc}%
A_{\left(  1\right)  } & B & 0\\
0 & 0_{m_{2}} & [I_{m_{2}}\;0]\\
0 & 0 & 0_{m_{1}}%
\end{array}
\right]  \text{.}\label{C=D=0_1}%
\end{equation}

\end{lemma}

\begin{proof}
Since $\operatorname{rank}E=m_{2}$, there is a nonsingular $V$ such that
$EV=[I_{m_{2}}\;0]$. For $S=I_{m-m_{1}}\oplus V$ we have%
\[
S^{\ast}\mathcal{A}S=\mathcal{A}S=\left[
\begin{array}
[c]{ccc}%
A_{\left(  1\right)  } & B & 0\\
C & D & [I_{m_{2}}\;0]\\
0 & 0 & 0_{m_{1}}%
\end{array}
\right]  :=\mathcal{A}^{\prime}\text{.}%
\]
Then, for%
\[
S=\left[
\begin{array}
[c]{cc}%
I_{m-m_{1}} & 0\\
X & I_{m_{1}}%
\end{array}
\right]  \text{\ and\ }X=-\left[
\begin{array}
[c]{cc}%
C & D\\
0 & 0
\end{array}
\right]  \text{,}%
\]
$\mathcal{A}^{\prime}S=S^{\ast}\mathcal{A}^{\prime}S$ has the form
(\ref{C=D=0_1}).
\end{proof}

\medskip A block matrix of the form (\ref{C=D=0_1}) is said to be a
\emph{*congruence reduced form} of $A$ if it is *congruent to $A$, $A_{\left(
1\right)  }$ is square, and $[A_{\left(  1\right)  }\;B]$ has linearly
independent rows. There are four possibilities for the $\rho$-by-$\rho$ matrix
$A_{\left(  1\right)  }$ in a *congruence reduced form of $A$:

\begin{itemize}
\item $\rho=0$: Then $A$ is *congruent to%
\[
\mathcal{A}=\left[
\begin{array}
[c]{cc}%
0_{m_{2}} & [I_{m_{2}}\;0]\\
0 & 0_{m_{1}}%
\end{array}
\right]  \text{.}%
\]
Since $\operatorname{rank}\mathcal{A}=m_{2}$ and $\mathcal{A}^{2}=0$, its
Jordan Canonical Form contains $m_{2}$ blocks $J_{2}$ and $m_{1}-m_{2}$ blocks
$J_{1}$. But $\mathcal{A}$ is similar to its Jordan Canonical Form via a
permutation similarity, which is a *congruence, so $J_{1}^{[m_{1}-m_{2}%
]}\oplus J_{2}^{[m_{2}]}$ is the regularizing decomposition for $A$.

\item $\rho>0$ and $A_{(1)}=0_{\rho}$, so $m_{3}=\operatorname{nullity}%
A_{(1)}=\rho$: $A$ is *congruent to%
\[
\mathcal{A}=\left[
\begin{array}
[c]{ccc}%
0_{m_{3}} & B & 0\\
0 & 0_{m_{2}} & [I_{m_{2}}\;0]\\
0 & 0 & 0_{m_{1}}%
\end{array}
\right]
\]
in which $B$ has full row rank. There is a nonsingular $V$ such that
$BV=[I_{m_{3}}\;0]$, so if we let $S=I_{m_{3}}\oplus V\oplus I_{m_{1}}$, we
have%
\[
S^{\ast}\mathcal{A}S=\left[
\begin{array}
[c]{ccc}%
0_{m_{3}} & [I_{m_{3}}\;0] & 0\\
0 & 0_{m_{2}} & [V^{\ast}\;0]\\
0 & 0 & 0_{m_{1}}%
\end{array}
\right]  :=R\text{.}%
\]
Now let $S=I_{m_{3}+m_{2}}\oplus(V^{-\ast}\oplus I_{m_{1}-m_{2}})$ and compute%
\[
S^{\ast}RS=\left[
\begin{array}
[c]{ccc}%
0_{m_{3}} & [I_{m_{3}}\;0] & 0\\
0 & 0_{m_{2}} & [I_{m_{2}}\;0]\\
0 & 0 & 0_{m_{1}}%
\end{array}
\right]  :=N\text{.}%
\]
Then $\operatorname{rank}N=m_{3}+m_{2}$, $\operatorname{rank}N^{2}=m_{3}$, and
$N^{3}=0$, so the Jordan Canonical Form of $N$ is $J_{1}^{[m_{1}-m_{2}]}\oplus
J_{2}^{[m_{2}-m_{3}]}\oplus J_{3}^{[m_{3}]}$, which is the regularizing
decomposition for $A$.

\item $\rho>0$ and $A_{(1)}$ is nonsingular: Let $R$ denote the block matrix
in (\ref{C=D=0_1}), let%
\[
S=\left[
\begin{array}
[c]{cc}%
I_{\rho} & -\left(  A_{\left(  1\right)  }\right)  ^{-1}B\\
0 & I_{m_{2}}%
\end{array}
\right]  \oplus I_{m_{1}}\text{,}%
\]
and compute%
\[
S^{\ast}RS=\left[
\begin{array}
[c]{ccc}%
A_{\left(  1\right)  } & 0 & 0\\
X & 0_{m_{2}} & [I_{m_{2}}\;0]\\
0 & 0 & 0_{m_{1}}%
\end{array}
\right]  \text{,}%
\]
in which $X=-B^{\ast}A_{(1)}^{-\ast}A_{(1)}$. Lemma \ref{lemma C=D=0} tells us
that $S^{\ast}RS$ is *congruent to (\ref{C=D=0_1}) with $B=0$, that is, to
$A_{\left(  1\right)  }\oplus M$ with
\[
M=\left[
\begin{array}
[c]{cc}%
0_{m_{2}} & [I_{m_{2}}\;0]\\
0 & 0_{m_{1}}%
\end{array}
\right]  \text{.}%
\]
Since $\operatorname{rank}M=m_{2}$ and $M^{2}=0$, the regularizing
decomposition of $A$ is $A_{\left(  1\right)  }\oplus J_{1}^{[m_{1}-m_{2}%
]}\oplus J_{2}^{[m_{2}]} $.

\item $\rho>0$ and $A_{(1)}$ is singular but nonzero:  We address this case in
the next lemma.
\end{itemize}

\begin{lemma}
\label{lemma singular A(1)}Let $\rho>0$ and let $A_{(1)}$ be the $\rho
$-by-$\rho$ upper left block in a *congruence reduced form (\ref{C=D=0_1}) of
$A$. Let $m_{3}$ and $m_{4}$ denote the nullity and *non-normal nullity,
respectively, of $A_{\left(  1\right)  }$, and suppose that $m_{3}>0$. Then
$A$ is *congruent to%
\begin{equation}
\left[
\begin{array}
[c]{ccccc}%
A_{\left(  2\right)  } & B^{\prime} & 0 & 0 & 0\\
0 & 0_{m_{4}} & [I_{m_{4}}\;0] & 0 & 0\\
0 & 0 & 0_{m_{3}} & [I_{m_{3}}\;0] & 0\\
0 & 0 & 0 & 0_{m_{2}} & [I_{m_{2}}\;0]\\
0 & 0 & 0 & 0 & 0_{m_{1}}%
\end{array}
\right]  \text{,}\label{singular A(1)_1}%
\end{equation}
in which $[A_{\left(  2\right)  }\;B^{\prime}]$ has linearly independent rows.
The parameters $m_{1}$, $m_{2}$, $m_{3}$, and $m_{4}$, and the *congruence
class of $A_{\left(  2\right)  }$ are *congruence invariants of $A$.
\end{lemma}

\begin{proof}
Step 1: Lemma \ref{lemma C=D=0} ensures that there is a nonsingular $S$ such
that
\[
S^{\ast}A_{(1)}S=\left[
\begin{array}
[c]{ccc}%
A_{(2)} & B^{\prime} & 0\\
0 & 0_{m_{4}} & [I_{m_{4}}\;0]\\
0 & 0 & 0_{m_{3}}%
\end{array}
\right]
\]
is a *congruence reduced form of $A_{(1)}$. Let $\rho^{\prime}$ denote the
size of $A_{(2)}$. Let $\hat{S}=S\oplus I_{m_{2}+m_{1}}$ and observe that
$\hat{S}^{\ast}A\hat{S}$ has the block form
\begin{equation}
\left[
\begin{array}
[c]{ccc}%
S^{\ast}A_{(1)}S & S^{\ast}B & 0\\
0 & 0_{m_{2}} & [I_{m_{2}}\;0]\\
0 & 0 & 0_{m_{1}}%
\end{array}
\right]  =\left[
\begin{array}
[c]{ccccc}%
A_{\left(  2\right)  } & B^{\prime} & 0 & B_{1} & 0\\
0 & 0_{m_{4}} & [I_{m_{4}}\;0] & B_{2} & 0\\
0 & 0 & 0_{m_{3}} & B_{3} & 0\\
0 & 0 & 0 & 0_{m_{2}} & [I_{m_{2}}\;0]\\
0 & 0 & 0 & 0 & 0_{m_{1}}%
\end{array}
\right] \label{singular A(1)_2}%
\end{equation}
in which%
\[
S^{\ast}B=\left[
\begin{array}
[c]{c}%
B_{1}\\
B_{2}\\
B_{3}%
\end{array}
\right]  \text{.}%
\]

Step 2: Let $M$ denote the upper left 2-by-3 block of the 5-by-5 block matrix
in (\ref{singular A(1)_2}). The rows of $M$ are linearly independent, so its
columns span $\mathbb{F}^{\rho^{\prime}+m_{4}}$. Add a linear combination of
the \emph{columns} of $M$ to the fourth block column of (\ref{singular A(1)_2}%
) in order to put zeros in the blocks $B_{1}$ and $B_{2}$. Complete this
column operation to a *congruence by adding the conjugate linear combination
of \emph{rows} of $M$ to the fourth block row of (\ref{singular A(1)_2}); this
spoils the zeros in the first four blocks of the fourth block row. Add linear
combinations of the fifth block column to the first four block columns in
order to re-establish the zero blocks there; the fifth block row is zero so
completing this column operation to a *congruence with a conjugate row
operation has no effect.

We have now achieved a *congruence of $A$ that has the form%
\begin{equation}
R=\left[
\begin{array}
[c]{ccccc}%
A_{\left(  2\right)  } & B^{\prime} & 0 & 0 & 0\\
0 & 0_{m_{4}} & [I_{m_{4}}\;0] & 0 & 0\\
0 & 0 & 0_{m_{3}} & B_{3} & 0\\
0 & 0 & 0 & 0_{m_{2}} & [I_{m_{2}}\;0]\\
0 & 0 & 0 & 0 & 0_{m_{1}}%
\end{array}
\right]  \text{,}\label{singular A(1)_2.3}%
\end{equation}
in which $B_{3}$ has linearly independent rows.

Step 3: Whenever one has a block matrix like that in (\ref{singular A(1)_2.3}%
), in which some of the superdiagonal blocks below the first block row do not
have the standard form $[I\;0]$ but nevertheless have linearly independent
rows, there is a finite sequence of *congruences that restores it to a
standard form like that in (\ref{singular A(1)_1}). For example, $B_{3}$ in
(\ref{singular A(1)_2.3}) has linearly independent rows, so there is a
nonsingular $V$ such that $B_{3}V=[I_{m_{3}}\;0]$. Right-multiply the 4th
block column of $R$ by $V$ and left-multiply the 4th block row of the result
by $V^{\ast}$. This restores the standard form of the block in position $3,4$
but spoils the $[I\;0]$ block in position $4,5$, though it still has linearly
independent rows. Now right-multiply the fifth block column by a factor that
restores it to standard form (in this case, the right multiplier is $V^{-\ast
}\oplus I_{m_{1}-m_{2}}$) and then left-multiply the fifth block row by the *
of that factor. If there are more than five block rows, continue this process
down the block superdiagonal to the block in the last block column, at which
point all of the superdiagonal blocks below the first block row are restored
to standard form since the last block row is zero. Of course, this finite
sequence of transformations is a *congruence of $R$.
\end{proof}

\medskip The preceding lemma clarifies the nature of the block $B$ in a
*congruence reduced form (\ref{C=D=0_1}) of $A$: except for the requirement
that $[A_{(1)}\;B]$ have full row rank, $B$ is otherwise arbitrary.

If there are different involutions on $\mathbb{F}$, the same matrix may have a
different regularizing decomposition for each involution. For example, take
$\mathbb{F}=\mathbb{C}$ and consider%
\[
A=\left[
\begin{array}
[c]{cc}%
1 & -i\\
i & 1
\end{array}
\right]  \text{.}%
\]
If the involution is complex conjugation, then $N(A)=N(A^{\ast})$ since $A$ is
Hermitian, $\zeta=m_{1}=1$, $\kappa=m_{2}=0$, and $\rho=1$; the regularizing
decomposition of $A$ is $[1]\oplus J_{1}$. However, if the involution is the
identity, then $N(A)\cap N(A^{T})=\{0\}$, $\zeta=m_{1}=0$, $\kappa=m_{2}=1$,
and $\rho=0$; the regularizing decomposition of $A$ is $J_{2}$.

\section{The Regularizing Decomposition\label{Section Decomposition}}

If the block $A_{\left(  2\right)  }$ in (\ref{singular A(1)_1}) is singular,
repeat the first two steps in the proof of Lemma \ref{lemma singular A(1)} to
reduce $A$ further by *congruence and produce the nullity $m_{5}$ and
*non-normal nullity $m_{6}$ of $A_{\left(  2\right)  }$ and a square matrix
$A_{(3)}$. Then perform the process described in Step 3 to restore the
standard form of the superdiagonal blocks below the first block row.

Reduction of $A$ to a sparse form that reveals all of its singular structure
under *congruence can be achieved by repeating the three steps in Lemma
\ref{lemma singular A(1)} to obtain successively smaller blocks $A_{(3)}$,
$A_{(4)}$, ..., $A_{(\tau)}$ (with successively smaller nullities) in which
$A_{(\tau)}$ is the first block that is nonsingular. The payoff for our effort
in deriving a form more sparse than that produced by the *congruences implicit
in the regularization algorithm alone, e.g., (\ref{Stage 2}), is that it
permits us to verify the validity of the regularizing decomposition asserted
in Theorem \ref{theorem regularization algorithm}.

\begin{theorem}
[Regularizing Decomposition]\label{theorem regularizing decomposition} Let $A$
be a given square singular matrix over $\mathbb{F}$. Perform the
regularization algorithm on $A $ and obtain the integers $\tau,m_{1}%
,m_{2},...,m_{2\tau}$ and a nonsingular matrix $A_{(\tau)}$. Then $\tau
,m_{1},m_{2},...,m_{2\tau}$ and the *congruence class of $A_{(\tau)}$ are
*congruence invariants of $A$. Moreover,%

\noindent
(a) (Canonical sparse form) $A$ is *congruent to $A_{(\tau)}\oplus N$, in
which
\begin{equation}
N=\left[
\begin{array}
[c]{cccccc}%
0_{m_{2\tau}} & [I_{m_{2\tau}}\;0] &  &  &  & \\
& 0_{m_{2\tau-1}} & [I_{m_{2\tau-1}}\;0] &  &  & \\
&  & \ddots & \ddots &  & \\
&  &  & 0_{m_{3}} & [I_{m_{3}}\;0] & \\
&  &  &  & 0_{m_{2}} & [I_{m_{2}}\;0]\\
&  &  &  &  & 0_{m_{1}}%
\end{array}
\right] \label{regularizing decomposition_1}%
\end{equation}
has all of its nonzero blocks $[I_{m_{2\tau}}\;0],\ldots,[I_{m_{2}}\;0]$ in
the first block superdiagonal, and each block $[I_{m_{k}}\;0]$ is $m_{k}%
$-by-$m_{k-1}$, $k=2,3,\ldots,2\tau$.%

\noindent
(b) (Existence)$\ A$ is *congruent to $A_{(\tau)}\oplus M$, in which%
\begin{equation}
M=J_{1}^{[m_{1}-m_{2}]}\oplus J_{2}^{[m_{2}-m_{3}]}\oplus J_{3}^{[m_{3}%
-m_{4}]}\oplus\cdots\oplus J_{2\tau-1}^{[m_{2\tau-1}-m_{2\tau}]}\oplus
J_{2\tau}^{[m_{2\tau}]}\text{.}\label{regularizing decomposition_2}%
\end{equation}
%

\noindent
(c) (Uniqueness) Suppose $A$ is *congruent to $B\oplus C$, in which $B$ is
nonsingular and $C$ is a direct sum of nilpotent Jordan blocks. Then $B$ is
*congruent to $A_{(\tau)}$ and some permutation of the direct summands of $C$
gives $M$.%

\noindent
(d) (Unitarily reduced form) Consider the three cases ($\alpha)$
$\mathbb{F=C}$ and the involution is complex conjugation, or ($\beta)$
$\mathbb{F=C}$ and the involution is the identity, or ($\gamma)$
$\mathbb{F=R}$ and the involution is the identity. Then depending on the case
there is ($\alpha)$ a complex unitary $U$, or ($\beta)$ a complex unitary $V$,
or ($\gamma)$ a real orthogonal $Q$ such that ($\alpha)$ $U^{\ast}AU$ or
($\beta)$ $V^{T}AV$ or ($\gamma)$ $Q^{T}AQ$ has the form
\begin{equation}
\left[
\begin{array}
[c]{ccccccccc|cc}%
B_{2\tau+1} & \multicolumn{1}{|c}{0} & 0 & \multicolumn{2}{|c|}{} &  &
\multicolumn{1}{c|}{} &  &  &  & \\\cline{1-3}%
\ast & \multicolumn{1}{|c}{\ast} & B_{2\tau} & \multicolumn{2}{|c|}{} &  &
\multicolumn{1}{c|}{} &  &  &  & \\
0 & \multicolumn{1}{|c}{0} & 0 & \multicolumn{2}{|c|}{} & \ast &
\multicolumn{1}{c|}{0} &  &  &  & \\\cline{1-3}
&  &  & \ddots &  & \multicolumn{2}{|c|}{} & \ast & 0 &  & \\[-6pt]
&  &  &  & \ddots & \multicolumn{1}{|c}{B_{7}} & 0 & \multicolumn{1}{|c}{} &
& \ast & 0\\\cline{1-7}%
\multicolumn{5}{c|}{\ast} & \ast & B_{6} & \multicolumn{1}{|c}{} &  &  & \\
\multicolumn{5}{c|}{0} & 0 & 0 & \multicolumn{1}{|c}{B_{5}} & 0 &  &
\\\cline{1-9}%
\multicolumn{7}{c|}{\ast} & \ast & B_{4} &  & \\
\multicolumn{7}{c|}{0} & 0 & 0 & B_{3} & 0\\\hline
\multicolumn{9}{c|}{\ast} & \ast & B_{2}\\
\multicolumn{9}{c|}{0} & 0 & 0
\end{array}
\right]  \!\!\!%
\begin{array}
[c]{l}%
\\
\}m_{2\tau}\\
\}m_{2\tau-1}\\
\\
\vdots\\
\vdots\\
\\
\}m_{4}\\
\}m_{3}\\
\}m_{2}\\
\}m_{1}%
\end{array}
\label{unitarycanonical}%
\end{equation}
in which all $2\tau+1$ diagonal blocks $B_{2\tau+1},\ast,0,\dots,\ast
,0,\ast,0$ are square, $B_{2\tau+1}$ is nonsingular, and each of $B_{2}%
,\dots,B_{2\tau}$ has linearly independent rows. The integers $\tau
,m_{1},\dots,m_{2\tau}$ are the same as those in
(\ref{regularizing decomposition_1}) and (\ref{regularizing decomposition_2}).
The equivalence class (under complex *congruence, complex $T$-congruence, or
real $T$-congruence, respectively) of $B_{2\tau+1}$ is the same as that of
$A_{(\tau)}$. In the principal submatrix of (\ref{unitarycanonical}) obtained
by deleting the block row and column containing $B_{2\tau+1}$, replacing all
blocks denoted by stars with zero blocks and replacing each $B_{i}$ with
$[I\ 0]$ produces the matrix $N$ in (\ref{regularizing decomposition_1}).
\end{theorem}

\begin{proof}
The *congruence invariance of the parameters $m_{i}$ and $\tau$, as well as
the *congruence class of $A_{(\tau)}$ have already been established. The form
of $N$ is the outcome of repeating the reduction described in Lemma
\ref{lemma singular A(1)} until it terminates with a block $A_{(\tau)}$ that
is nonsingular. The only issue is the explicit description of the Jordan block
structure in (\ref{regularizing decomposition_2}).

Notice that%
\[
\underset{m_{k-1}}{\underbrace{[I_{m_{k}}\;0]}}\underset{m_{k-2}}%
{\underbrace{\,[I_{m_{k-2}}\;0]}}=\underset{m_{k-2}}{\underbrace{[I_{m_{k}%
}\;0]}}%
\]
and hence
\[
N^{2}=\left[
\begin{array}
[c]{cccccc}%
0_{m_{2\tau}} & 0 & [I_{m_{2\tau}}\;0] &  &  & \\
& 0_{m_{2\tau-1}} & 0 & [I_{m_{2\tau-1}}\;0] &  & \\
&  & \ddots & \ddots & \ddots & \\
&  &  & 0_{m_{3}} & 0 & [I_{m_{3}}\;0]\\
&  &  &  & 0_{m_{2}} & 0\\
&  &  &  &  & 0_{m_{1}}%
\end{array}
\right]
\]
has its nonzero blocks $[I_{m_{2\tau}}\;0],\ldots,[I_{m_{3}}\;0]$ in the
\emph{second} block superdiagonal. In general, $N^{k}$ is a $0$-$1$ matrix
that has its nonzero blocks $[I_{m_{2\tau}}\;0],\ldots,[I_{m_{k+1}}\;0]$ in
the $k\emph{th}$ block superdiagonal. The structure of the powers $N^{k}$
ensures that the rank of each is equal to the number of its nonzero entries,
so%
\begin{equation}
\operatorname{rank}N^{k}=m_{k+1}+\cdots+m_{2\tau}\text{,\quad\ }%
\kappa=1,...,2\tau-1\label{proof regularizing decomposition_1}%
\end{equation}
and $N^{2\tau}=0$. The list of multiplicities of the nilpotent Jordan blocks
in the Jordan Canonical Form of $N$ (arranged in order of increasing size) is
given by the sequence of second differences of the sequence $\left\{
\operatorname{rank}N^{k}\right\}  _{k=1}^{2\tau}$ \cite[Exercise, p. 127]%
{hor}, which is $m_{1}-m_{2}$, $m_{2}-m_{3}$, $m_{3}-m_{4}$, etc. The direct
sum of nilpotent Jordan blocks in (\ref{regularizing decomposition_2}) is
therefore the Jordan Canonical Form of $N$.

The final step in proving (\ref{regularizing decomposition_2}) is to show that
the Jordan Canonical Form of $N$ can be achieved via a permutation similarity,
which is a *congruence. A conceptual way to do this is to show that the
directed graphs of the two matrices $M$ and $N$ are isomorphic.

The directed graph of $J_{k}$ is a linear chain with $k$ nodes $P_{1}%
,\ldots,P_{k}$ in which there is an arc from $P_{i}$ to $P_{i+1}$ for each
$i=1,\ldots,k-1$, so the directed graph of $M$ is a disjoint union of such
linear chains. There are $m_{k}-m_{k+1}$ chains with $k$ nodes for each
$k=1,\ldots,2\tau$.

To understand the directed graph of $N$ one can begin with any node
corresponding to any row in the first block row. Each of these $m_{2\tau}$
nodes is the first in a linear chain with $2\tau$ nodes. In the second block
row, the nodes corresponding to the first $m_{2\tau}$ rows are members of the
linear chains associated with the first block row, but the nodes corresponding
to the last $m_{2\tau-1}-m_{2\tau}$ rows begin their own linear chains, each
with $2\tau-1$ nodes. Proceeding in this way downward through the block rows
of $N$ we identify a set of disjoint linear chains that is identical to the
set of disjoint linear chains associated with $M$. A permutation of labels of
nodes that identifies the directed graphs of $M$ and $N$ gives a permutation
matrix that achieves the desired permutation similarity between $M$ and $N$.

The uniqueness assertion follows from (a) our identification of all the
relevant parameters as *congruence invariants of $A$ and (b) uniqueness of the
Jordan Canonical Form.

Finally, the assertions about the unitarily reduced form
(\ref{unitarycanonical}) follow from the regularizing algorithm in Section
\ref{Section Algorithm} and the proof of Theorem
\ref{Theorem Unitary transformations}. When the regularizing algorithm is
carried out with unitary transformations, the result is a matrix of the form
(\ref{unitarycanonical}), of which (\ref{Stage 2}) is a special case.
\end{proof}

\medskip

\section{Regularization of a *Selfadjoint Pencil\label{Section Pencil}}

Theorem \ref{theorem regularizing decomposition} implies that every
*selfadjoint matrix pencil $A+\lambda A^{\ast}$ has a regularizing
decomposition \eqref{eq2} with a *selfadjoint regular part. The algorithm in
Section \ref{Section Algorithm} can be used to construct the regularizing
decomposition, and if $\mathbb{F}=\mathbb{C}$ with either the identity or
complex conjugation as the involution (respectively, $\mathbb{F}=\mathbb{R}$
with the identity involution), the construction can be carried out using only
unitary (respectively, real orthogonal) transformations. We emphasize that the
involution on $\mathbb{F}$ may be the identity, so the assertions in the
following theorem are valid for matrix pencils of the form $A+\lambda A^{T}$.

\begin{theorem}
\label{penc} Let $A+\lambda A^{\ast}$ be a *\negthinspace selfadjoint matrix
pencil over $\mathbb{F}$ and let $A$ be *congruent to $A_{(\tau)}\oplus M$, in
which $A_{(\tau)}$ is nonsingular and $M$ is the direct sum of nilpotent
Jordan blocks in (\ref{regularizing decomposition_2}). Then there is a
nonsingular $S$ such that $S(A+\lambda A^{\ast})S^{\ast}=(A_{(\tau)}+\lambda
A_{(\tau)}^{\ast})\oplus K$ and
\[
K=(J_{1}+\lambda J_{1}^{T})^{[m_{1}-m_{2}]}\oplus(J_{2}+\lambda J_{2}%
^{T})^{[m_{2}-m_{3}]}\oplus\cdots\oplus(J_{2\tau}+\lambda J_{2\tau}%
^{T})^{[m_{2\tau}]}\text{.}%
\]
Moreover, each singular block $J_{k}+\lambda J_{k}^{T}$ may be replaced by
\begin{equation}%
\begin{cases}
(F_{\ell}+\lambda G_{\ell})\oplus(G_{\ell}^{T}+\lambda F_{\ell}^{T}) &
\text{if }k=\text{$2\ell-1$ is odd}\\
\big(J_{\ell}+\lambda I_{\ell}\big)\oplus\big(I_{\ell}+\lambda J_{\ell}\big) &
\text{if $k=2$}\ell\text{ is even.}%
\end{cases}
\label{pen1}%
\end{equation}

\end{theorem}

Use of the blocks (\ref{pen1}) instead of the corresponding Jordan blocks is
justified by the following lemma.

\begin{lemma}
\label{lem13} $J_{k}+\lambda J_{k}^{T}$ is strictly equivalent to (\ref{pen1}).
\end{lemma}

\begin{proof}
If there is a permutation matrix $S$ such that
\[
SJ_{k}S^{T}=%
\begin{cases}
M_{\ell}:=%
\begin{bmatrix}
0 & G_{\ell}^{T}\\
F_{\ell} & 0
\end{bmatrix}
& \text{if }k=2\ell-1\text{ is odd}\\[1.2em]%
N_{\ell}:=%
\begin{bmatrix}
0 & I_{\ell}\\
J_{m} & 0
\end{bmatrix}
& \text{if }k=2\ell\text{ is even,}%
\end{cases}
\]
then $S(J_{k}+\lambda J_{k}^{T})S^{T}$ is strictly equivalent to (\ref{pen1}).
To prove the existence of such an $S$, we need to prove that $M_{\ell}$ and
$N_{\ell}$ can be obtained from $J_{k}$ by simultaneous permutations of rows
and columns, that is, there exists a permutation $f$ on $\{1,2,\dots,k\}$ that
transforms the positions
\[
(1,2),\ (2,3),\,\dots,\,(k-1,k)
\]
of the unit entries in $J_{k}$ to the positions
\begin{equation}
(f(1),f(2)),\ \ (f(2),f(3)),\ \dots,\ (f(k-1),f(k))\label{wqw}%
\end{equation}
of the unit entries in $M_{\ell}$ if $k=2\ell-1$ or in $N_{\ell}$ if $k=2\ell
$. To obtain the sequence (\ref{wqw}), we arrange the indices of the units in
\[
M_{\ell}=\left[
\begin{array}
[c]{c|c}%
\text{{\LARGE 0}} &
\begin{matrix}
0 &  & 0\\
1 & \ddots & \\
& \ddots & 0\\
0 &  & 1
\end{matrix}
\\\hline%
\begin{matrix}
1 & 0 &  & 0\\
& \ddots & \ddots & \\
0 &  & 1 & 0
\end{matrix}
& \text{{\LARGE 0}}%
\end{array}
\right]  \quad(\text{$(2\ell-1)$-by-$(2\ell-1)$})
\]
as follows:
\begin{multline*}
(\ell,2\ell-1),\ (2\ell-1,\ell-1),\ (\ell-1,2\ell-2),\ (2\ell-2,\ell-2),\\
\dots,(2,\ell+1),\ (\ell+1,1),
\end{multline*}
and the indices of the units in $N_{k}$ as follows:
\[
(1,\ell+1),\,(\ell+1,2),\,(2,\ell+2),\,(\ell+2,3),\dots,(2\ell-1,\ell
),\,(\ell,2\ell)\text{.}%
\]

\end{proof}


\begin{thebibliography}{99}                                                                                               %
\bibitem {DSZ}D. \v{Z}. \raisebox{1pt}{-}\negthinspace\negthinspace
Dokovi\'{c}, F. Szechtman, and K. Zhao, An algorithm that carries a square
matrix into its transpose by an involutory congruence transformation,
\emph{Electron. J. Linear Algebra} 10 (2003) 320-340.

\bibitem {fro}F. G. Frobenius, \emph{Gesammelte Abhandlungen}, Vol. I,
Springer, Heidelberg, 1968.

\bibitem {gab}P. Gabriel, Appendix: degenerate bilinear forms, \textit{J.
Algebra} 31 (1974) 67--72.

\bibitem {gan}F. R. Gantmacher, \textit{The Theory of Matrices}, Chelsea, New
York, 2000.

\bibitem {hor}R. A. Horn and C. R. Johnson, \textit{Matrix Analysis},
Cambridge University Press, New York, 1987.

\bibitem {hor-ser}R. A. Horn and V. V. Sergeichuk, Congruence of a square
matrix and its transpose, \textit{Linear Algebra Appl.} 389 (2004) 347--353.

\bibitem {Riehm}C. Riehm and M. Shrader-Frechette, The equivalence of
sesquilinear forms, \emph{J. Algebra} 42 (1976) 495-530.

\bibitem {ser1}V. V. Sergeichuk, Classification problems for system of forms
and linear mappings, \textit{Math. USSR, Izvestiya}\/ 31 (3) (1988) 481--501.

\bibitem {vvs2004}V. V. Sergeichuk, Computation of canonical matrices for
chains and cycles of linear mappings, \emph{Linear Algebra Appl.} 376 (2004) 235-263.

\bibitem {doo}P. Van Dooren, The computation of Kronecker's canonical form of
a singular pencil, \textit{Linear Algebra Appl.} 27 (1979) 103--140.
\end{thebibliography}
\end{document}